\documentclass[reqno, final,a4paper]{amsart}

\usepackage{color}
\usepackage[colorlinks=true,allcolors=blue
]{hyperref}
\usepackage{amsmath, amssymb, amsthm}
\usepackage{mathrsfs}
\usepackage{mathtools}
\usepackage[noabbrev,capitalize,nameinlink]{cleveref}
\crefname{equation}{}{}
\usepackage{fullpage}
\usepackage[noadjust]{cite}
\usepackage{graphics}
\usepackage{pifont}
\usepackage{tikz}
\usepackage{bbm}
\usepackage[T1]{fontenc}
\usetikzlibrary{arrows.meta}

\usepackage{environ}
\usepackage{framed}
\usepackage{url}
\usepackage[linesnumbered,ruled,vlined]{algorithm2e}
\usepackage[noend]{algpseudocode}
\usepackage[labelfont=bf]{caption}
\usepackage{framed}
\usepackage[framemethod=tikz]{mdframed}
\usepackage{pgfplots}
\usepackage{appendix}
\usepackage{graphicx}
\usepackage[textsize=tiny]{todonotes}
\usepackage{tcolorbox}
\usepackage{xcolor}
\usepackage[shortlabels]{enumitem}
\usepackage{cancel}
\crefformat{enumi}{#2#1#3}
\crefrangeformat{enumi}{#3#1#4 to~#5#2#6}
\crefmultiformat{enumi}{#2#1#3}
{ and~#2#1#3}{, #2#1#3}{ and~#2#1#3}
\allowdisplaybreaks
\usepackage{paralist}

\let\originalleft\left
\let\originalright\right
\renewcommand{\left}{\mathopen{}\mathclose\bgroup\originalleft}
\renewcommand{\right}{\aftergroup\egroup\originalright}

\let\originalbig\big
\renewcommand{\big}{\mathopen{}\mathclose\originalbig}

\crefname{equation}{}{} 
\AtBeginEnvironment{appendices}{\crefalias{section}{appendix}} 

\colorlet{refkey}{orange!20}
\colorlet{labelkey}{blue!30}

\crefname{algocf}{Algorithm}{Algorithms}

\numberwithin{equation}{section}
\newtheorem{theorem}{Theorem}[section]
\newtheorem{proposition}[theorem]{Proposition}
\newtheorem{lemma}[theorem]{Lemma}
\newtheorem{claim}[theorem]{Claim}
\crefname{claim}{Claim}{Claims}

\newtheorem{corollary}[theorem]{Corollary}

\newtheorem*{question*}{Question}
\newtheorem{fact}[theorem]{Fact}

\theoremstyle{definition}
\newtheorem{definition}[theorem]{Definition}

\newtheorem*{definition*}{Definition}
\newtheorem*{theorem*}{Theorem}

\theoremstyle{remark}

\newtheorem*{remark*}{Remark}

\newcommand{\mb}{\mathbb}

\newcommand{\mc}{\mathcal}

\newcommand{\mr}{\mathrm}

\newcommand{\on}{\operatorname}

\newcommand{\N}{\mathbb{N}}

\newcommand{\eps}{\varepsilon}
\newcommand{\E}{\mathbb{E}}
\renewcommand{\Pr}{\mb P}

\newcommand{\su}{\subseteq}

\newcommand{\hide}[1]{}

\allowdisplaybreaks

\title{Exponential anticoncentration of the permanent}

\author[Hunter]{Zach Hunter}

\address{Department of Mathematics, ETH Z\"urich, Switzerland.}
\email{zach.hunter@ifor.math.ethz.ch}

\author[Kwan]{Matthew Kwan}

\address{Institute of Science and Technology Austria (ISTA).}
\email{matthew.kwan@ist.ac.at}

\author[Sauermann]{Lisa Sauermann}

\address{Institute for Applied Mathematics, University of Bonn, Germany.}
\email{sauermann@iam.uni-bonn.de}

\thanks{
Zach Hunter was supported by SNSF grant 200021-228014. Matthew Kwan was supported by ERC Starting Grant ``RANDSTRUCT'' No.\ 101076777. Lisa Sauermann was supported by the
DFG Heisenberg Program. 
}

\begin{document}

\maketitle
\begin{abstract}
Let $A\in\mathbb{R}^{n\times n}$ be a random matrix with independent
entries, and suppose that the entries are ``uniformly anticoncentrated''
in the sense that there is a constant $\eps>0$ such that each
entry $a_{ij}$ satisfies $\sup_{z}\Pr[a_{ij}=z]\le1-\eps$
(for example, $A$ could be a uniformly random $n\times n$ matrix
with $\pm1$ entries). Significantly improving previous bounds of Tao and Vu, we prove that the permanent of $A$ is exponentially anticoncentrated: there is $c_{\eps}>0$ such that $\sup_{z}\Pr[\operatorname{per}(A)=z]\le\exp(-c_{\eps}n)$. Our proof also
works for the determinant, giving an alternative proof of a classical
theorem of Kahn, Koml\'os and Szemer\'edi. As a consequence, we see that there are at least exponentially many
different permanents of $n\times n$ matrices with $\pm1$ entries,
resolving a problem of Ingram and Razborov.
\end{abstract}

\section{Introduction}

Two important matrix parameters are the \emph{determinant}
and \emph{permanent}: for an $n\times n$ matrix $A=(a_{ij})$, 
they are defined as
\begin{equation}
\det(A)=\sum_{\pi\in S_{n}}\operatorname{sign}(\pi)\prod_{i=1}^{n}a_{i\pi(i)}\qquad\text{and}\qquad\operatorname{per}(A)=\sum_{\pi\in S_{n}}\prod_{i=1}^{n}a_{i\pi(i)}.\label{eq:det-per}
\end{equation}
Determinants are omnipresent throughout mathematics and the sciences. Permanents are of fundamental importance in combinatorics and computational complexity theory, and also play an important role in quantum physics and statistics (see for example the books and surveys \cite{ABbook,vLW01,Minbook,Sch04,MR21}).

In combinatorics, computer science and random matrix theory, it is of great interest to understand statistical properties of determinants and permanents, among various classes of discrete matrices. 
In particular, one of the most intensive directions of
research is the study of the singularity probability of random sign matrices. Let $A_{n}$
be a uniformly random $n\times n$ matrix with $\pm1$ entries; what
is the probability that $\det(A_{n})=0$? Perhaps the foundational
theorem in combinatorial random matrix theory, due to Koml\'os~\cite{Kom67},
is that $\Pr[\det(A_{n})=0]=o(1)$ as $n\to\infty$. A subsequent
breakthrough by Kahn, Koml\'os and Szemer\'edi~\cite{KKS95} (simplified
by Tao and Vu~\cite{TV06b}) showed that the singularity probability
is in fact \emph{exponentially small}: that is to say, $\Pr[\det(A_{n})=0]=\exp(-cn)$
for some constant $c>0$. Quantitative aspects were later improved
by Tao and Vu~\cite{TV07} and by Bourgain, Vu and Wood~\cite{BVW10},
culminating in a celebrated result of Tikhomirov~\cite{Tik20} proving the essentially optimal bound $\Pr[\det(A_{n})=0]\le(1/2+o(1))^{n}$.

It is reasonable to expect similar results for the permanent. In fact one may expect that \emph{even stronger} bounds
should hold; for example, Vu (see \cite[Conjecture~6.12]{Vu20} and
\cite{Vu09}) conjectured that $\Pr[\on{per}(A_{n})=0]$ decays faster than exponentially. 
However, the permanent is much more difficult to study than the determinant,
and much less is known: until now, the only theorem in this
direction is due to Tao and Vu~\cite{TV09}, who proved that
\[
\Pr[\on{per}(A_{n})=0]\le n^{-c}
\]
for some constant $c>0$ 
(see also \cite{TV08} for an earlier, simpler version of their argument that
gives a bound of shape $\exp(-c\sqrt{\log n})$). Although Tao and Vu did not attempt to optimise the constant $c$, they note that their argument cannot hope to obtain a stronger bound than about $1/\sqrt n$, owing to their use of the so-called \emph{Erd\H os--Littlewood--Offord inequality} (they write that ``in principle, one can obtain better results by using more advanced Littlewood-Offord inequalities, but it is not clear to the authors how to restructure the rest of the argument so that such inequalities can be exploited'').

In this paper, we overcome this barrier, obtaining the first exponential bound.
\begin{theorem}\label{thm:baby-main}
There is an absolute constant $c>0$ such that the following holds. Let $A_n$ be a uniformly random $n\times n$ matrix with $\pm1$ entries. Then, for any $x\in \mb R$,
\[\Pr[\on{per}(A_n)=x]\le \exp(-c n).\]
\end{theorem}

Here, we stated a bound for the probability that $\on{per}(A_n)=x$ for general $x$, not just for $x=0$, but this should not be viewed as a new feature of our work: it is easy to make the same generalisation with all previous work on determinants and permanents of discrete random matrices.

It is worth remarking that our proof works in exactly the same way
for the determinant instead of the permanent, so as a byproduct this
also gives a new proof of the Kahn--Koml\'os--Szemer\'edi theorem.
All previous work on determinants of random matrices fundamentally used linear-algebraic properties of the determinant; since such properties are unavailable
for the permanent, our proof is necessarily very different. In fact, our proof is also very different to the previous approach of Tao and Vu in \cite{TV09} (which was until now the only known approach to study permanents of random matrices).

\subsection{General entry distributions}

Some of the previous work on determinants and permanents of discrete random matrices readily generalises to other entry distributions. For example, shortly after Koml\'os' foundational paper showing that $\Pr[\det(A_n)=0]=o(1)$, in a second paper~\cite{Kom68} he proved that the same result holds for a random matrix with independent $\mu$-distributed entries, for any non-degenerate distribution $\mu$ (fixing $\mu$ and sending $n\to \infty $).

Bourgain, Vu and Wood~\cite{BVW10} observed that in fact one does not need to fix the entry distribution; it suffices for the entries to be independent and ``uniformly anticoncentrated'' in the sense that there is a constant $\eps>0$ such that each entry $a_{ij}$ satisfies $\sup_{z}\Pr[a_{ij}=z]\le1-\eps$. Specifically, they observed (see \cite[Theorem~1.4]{BVW10}) that some of their ideas could be combined with an approach of Tao and Vu~\cite{TV06b}, to prove that $\Pr[\det(A_n)=0]$ decays exponentially in this very general setting. Here, we prove the same statement for the permanent.

For a random variable $X$, we write $Q[X]=\sup_{x}\Pr[X=x]$ for the maximum point probability of $X$.

\begin{theorem}
\label{thm:main}For any $0<p<1$ there is $c_{p}>0$ such that the
following holds. Let $A_n\in\mb{R}^{n\times n}$ be a random matrix with independent entries, such that every entry $\xi$ of $A_n$ satisfies $Q[\xi]\le p$. 
Then
\[
Q[\on{per}(A_n)]\le\exp(-c_{p}n).
\]
\end{theorem}
Note that \cref{thm:baby-main} is an immediate corollary of \cref{thm:main} for $p=1/2$. We remark that while \cref{thm:baby-main} might be improvable to a super-exponential bound (as conjectured by Vu), in the more general setting of \cref{thm:main} an exponential bound is best-possible. Indeed, if the entries in the first row of $A_n$ are each equal to zero with probability $p$, then with probability $p^n$ the entire first row is zero, in which case $\on{per}(A_n)=0$.

\subsection{The range of the permanent}
For any set $S\subseteq \mb R$ and any $n\in\mb N$, we write $\Phi_S(n)$ for the set of all permanents of $n\times n$ matrices with entries in $S$.

For any discrete random variable $X$, it is easy to see that $X$ must be supported on at least $1/Q[X]$ different values. So, by applying \cref{thm:main} to random matrices whose entries are uniformly sampled from any pair $\{a,b\}\subseteq \mb R$, we obtain an exponential lower bound on $|\Phi_{\{a,b\}}(n)|$.
\begin{corollary}\label{cor:exponential-range}
For any set $S\subseteq \mb R$ of size at least 2, taking $c_{1/2}$ as in \cref{thm:main}, we have
\[|\Phi_S(n)|\ge \exp(c_{1/2}n),\]
\end{corollary}

This result is not new when $S=\{0,1\}$. Indeed, in this case a classical theorem of Brualdi and Newman~\cite{BN65} says that $\Phi_{\{0,1\}}(n)\supseteq\{0,\dots,2^{n-1}\}$ (see also the improvements in \cite{GT18}). However, \cref{cor:exponential-range} significantly improves on previous results when $S=\{-1,1\}$. Indeed, in 1993 Kr\"auter~\cite{Kra93} showed that $|\Phi_{\{-1,1\}}(n)|\ge n+1$. This has remained essentially the state of the art\footnote{One can slightly improve Kr\"auter's bound using Tao and Vu's results on permanents of random matrices, but it is unclear how to get more than $o(n)$ additional permanents this way.} until very recently, when Ingram and Razborov~\cite{IR} showed how to use ideas from additive combinatorics and Diophantine geometry to prove that $|\Phi_{\{-1,1\}}(n)|\ge\exp(c(\log n)^{2}/\log\log n)$ for some constant $c>0$. Actually, they explicitly raised the problem (see \cite[Problem 1]{IR}) of proving that $|\Phi_{\{-1,1\}}(n)|$ grows exponentially in $n$; this problem is solved by \cref{cor:exponential-range}.

\subsection{Proof ideas}
The full proof of \cref{thm:main} is not very long, but it nonetheless seems worthwhile to give a brief glimpse of the ideas. The starting point is a ``relative anticoncentration inequality'' (\cref{lem:relative-halasz}), proved using a theorem of Kesten~\cite{Kes69} (which appears as \cref{thm:kesten}, and is the only aspect of our proof which is not self-contained). Very roughly speaking, \cref{lem:relative-halasz} says that, under certain assumptions, sums of independent random variables are more anticoncentrated than ``thinned-out versions of themselves'' obtained by taking a randomly subsampled subset of the summands. This inequality is related to certain inequalities of Hal\'asz~\cite{Hal77} that played a crucial role in the work of Kahn, Koml\'os and Szemer\'edi~\cite{KKS95} on the singularity probability (mentioned earlier in this introduction). However, the way we use \cref{lem:relative-halasz} is rather different to all previous work in this area: we use it to prove a \emph{recursive} anticoncentration bound, taking advantage of the recursive properties of the formula in \cref{eq:det-per}.

In a bit more detail: in the setting of \cref{thm:main}, it is not hard to see that $\on{per}(A_n)$ can be interpreted as a weighted sum $a_{1}\xi_{1}+\dots+a_{n}\xi_{n}$,
where $\xi_{1},\dots,\xi_{n}$ are the entries of the last row, and
$a_{1},\dots,a_{n}$ are permanents of $(n-1)\times(n-1)$ submatrices of the first $n-1$ rows. If we condition on the first $n-1$ rows of $A_n$ (so the only remaining randomness is in the last row) then this is a sum of independent random variables, to which we can apply our relative anticoncentration theorem. It turns out that if we average the resulting bound over all possible outcomes of the first $n-1$ rows, we get a bound on the anticoncentration of $\on{per}(A_n)$, in terms of  the anticoncentration of the permanents of certain ``smaller'' random matrices.

To see why this is useful, suppose that we were able to prove a bound of the form $Q[\on{per}(A_n)]\le \alpha_pQ[\on{per}(A_{n-1})]$, for some $\alpha_p<1$. We would then be able to iterate this bound to prove that $Q[\on{per}(A_n)]\le \alpha_p^n$; i.e., $\on{per}(A_{n})$ is exponentially anticoncentrated, as desired. Unfortunately, the situation is not quite this simple: it turns out that for our smaller random matrices, 
the point probabilities of the entries are not bounded by $p$ anymore,
and also there are ``error terms'' corresponding to certain assumptions we need to make to effectively apply \cref{lem:relative-halasz}. Nonetheless, with various additional ideas, we can prove a sufficiently strong recursive anticoncentration bound (see \cref{thm:inductive})
for a certain more general class of random matrices (which are allowed to contain some fixed deterministic rows in addition to their random rows).

It is worth remarking that our proof is ``purely combinatorial''; in particular, it features no geometry, no linear algebra and no Fourier analysis. 

\subsection{Further directions}

The most obvious direction for further research is to improve the
bound in \cref{thm:baby-main}, towards the super-exponential
bound conjectured by Vu.

Also, it would be very interesting to generalise these ideas to \emph{symmetric} random matrices. For a random symmetric $n\times n$ matrix $S_n$ with $\pm 1$ entries, the second and third author~\cite{KS22} proved that $\Pr[\on{per}(S_n)=0]\le n^{-c}$ for some constant $c$, but this method cannot prove a bound stronger than $1/\sqrt n$. Even for the determinant, it is much more difficult to prove exponential bounds in the symmetric setting: this was accomplished only recently by Campos, Jenssen, Michelen and Sahasrabudhe~\cite{CJMS25}.

Another direction is to consider the ``small-ball'' setting: namely, instead of just studying the probability that $\on{per} (A_{n})$ is equal to a given point, we could study the probability that it lies in an interval of given length. Actually, Tao and Vu~\cite{TV08} already considered this small-ball setting: their proof incorporates a ``growth'' step, which allows them to prove that $\Pr[|\on{per} (A_{n})-x|\le n^{(1/2-\eps)n}]\le n^{-1/10}$ for any $x\in \mb R$ and any constant $\eps>0$ (assuming $n$ is sufficiently large). By incorporating this growth step into our own proof scheme, one can prove the same result with an exponential probability bound. (This introduces some additional technicalities, so in the interest of keeping the arguments simple, in this paper we decided to restrict our attention to the point-probability setting.)

In the latter direction it is worth mentioning the \emph{permanent anticoncentration conjecture} of Aaronson and Arkhipov~\cite{AA13}
(a key ingredient in the analysis
of the hardness of the \emph{BosonSampling} problem in quantum computing).
This is a precise conjecture about the typical size of the permanent of a Gaussian random matrix,
seemingly well outside the reach of current techniques. See
\cite{AA13} for some discussion of the relevant obstacles; some of
these obstacles can be overcome by ideas in this paper, but the critical
obstacle is that existing techniques are not sufficiently sensitive
to small differences in the permanent.

\subsection{Notation}
Let $\N=\{0,1,2,\dots\}$ be the set of non-negative integers. We use asymptotic notation throughout: for functions $f=f(n)$ and $g=g(n)$, we write $f=O(g)$ or $g=\Omega(f)$ or $f\lesssim g$ to mean that there is a constant $C$ such that $|f(n)|\le C|g(n)|$ for sufficiently large $n$, 
and we write $f=o(g)$ to mean that $f/g\to0$ as $n\to\infty$. Subscripts on asymptotic notation indicate quantities that should be treated as constants. Also, all random variables throughout this paper are assumed to be real-valued.

\section{A recursive bound}

We will actually prove a slight strengthening of \cref{thm:main}, where we consider matrices which have a random part and a fixed part. This is necessary for our recursive approach.

\begin{definition}\label{def:A}
For $0<p<1$ and $n\in \mb N$, let $f_p(n)$ be the supremum of $Q[\on{per}(A)]$, over
all random matrices $A$ of the following form.
    \begin{compactitem}
\item For some  $T\in \mb N$, let $A_{\mathrm{fix}}\in\mb R^{T\times(n+T)}$ be a fixed $T\times (n+T)$
matrix containing at least one $T\times T$ submatrix with nonzero
permanent.
\item Let $A_{\mathrm{rand}}\in\mb R^{n\times(n+T)}$ be a random matrix with independent entries, such that every entry $\xi$ of $A_{\mathrm{rand}}$ satisfies $Q[\xi]\le p$.
\item Let $A\in\mb R^{(n+T)\times (n+T)}$ be the $(n+T)\times (n+T)$ matrix obtained by
concatenating $A_{\mathrm{fix}}$ and $A_{\mathrm{rand}}$ (i.e.,
the first $T$ rows are the rows of $A_{\mathrm{fix}}$, and the last
$n$ rows are the rows of $A_{\mathrm{rand}}$).
\end{compactitem}
\end{definition}
In \cref{def:A}, we think of $n$ as being the ``size'' of $A$ (the random part $A_{\mr{rand}}$ has $n$ rows, and in the proof of \cref{thm:main}, we will only really use the randomness of $n$ of its columns). We will prove a recursive bound for $f_p(n)$, see \cref{thm:inductive} below. To state this bound, for any $0<p<1$, let us define 
\begin{equation}
\tau_p= 1-\prod_{s=1}^\infty (1-p^s),\label{eq:def:tau-p}
\end{equation}
and observe that $\tau_p<1$, no matter how close $p$ is to $1$. To see this, let $C_p$ be sufficiently large in terms of $p$ such that $\exp(-C_p p)\le 1-p$, so for every $0\le x\le p$ we have $1-x\ge \exp(-C_px)$, and hence $1-\tau_p=\prod_{s=1}^\infty (1-p^s)\ge \prod_{s=1}^\infty \exp(-C_p p^s)=\exp\big(\!-C_p\sum_{s=1}^\infty p^s\big)=\exp(-C_pp/(1-p))>0$.

\begin{theorem}
\label{thm:inductive}For any $0<p<1$ and $t\in \mb N$ and $\eps>0$, there is $\delta>0$
and $N\in \mb N$ such that for all $n\ge N$ we have
\[
f_p(n)\le\exp(-\delta n)\;+\;\eps \sum_{s=1}^{t}p^sf_p(n-s)\;+\;(1+\eps)\tau_p \cdot f_p(n-t)\;+\;(1+\eps)\sum_{s=t+1}^{n}p^{s}f_p(n-s).
\]
\end{theorem}
In the course of proving \cref{thm:inductive}, we will also prove the following easier bound (which shall serve as a base case when proving \cref{thm:main} by induction).
\begin{proposition}\label{prop:easy}
    For any $p<1$ and $n\in \mb N$, we have
    \[f_p(n)\le 1-\prod_{s=1}^n(1-p^s)\le \tau_p<1.
    \]
\end{proposition}

\begin{proof}[Proof of \cref{thm:main} given \cref{thm:inductive} and \cref{prop:easy}]
As in the theorem statement, fix $0<p<1$. First, note that it suffices to prove that $f_p(n)\le \exp(-c_p n)$ for some $c_p>0$. Indeed, any random matrix $A_n$ as in \cref{thm:main} trivially satisfies the conditions in \cref{def:A} for $T=0$, so $Q[\on{per}(A_n)]\le f_p(n)$.

Recalling that $\tau_p<1$, choose $\mu>0$ such that $\tau_p<1-5\mu$. Then,  choose $\eps>0$ such that $\eps<(1-\sqrt{p})\mu$ and $(1+\eps)\tau_p<(1+\eps)(1-5\mu)<1-4\mu$. Also, choose $t\in \mb N$ such that $p^{t/2}<(1-\sqrt{p})\mu$. Now, let $\delta>0$ and $N\in \mb N$ be as in \cref{thm:inductive}. By making $N$ larger if needed, we may assume that $\exp(-\delta N/2)\le \mu$.

By \cref{prop:easy} we have $f_p(n)<1$ for all $n\in \mb N$. So, we can choose some $c_p>0$ such that $f_p(n)<\exp(-c_p n)$ for $n=0,1,\dots,N$, and also $c_p<\delta/2$ and $\exp(c_p)<p^{-1/2}$ and $(1-4\mu)\exp(c_p t)<1-3\mu$.

Let us now prove by induction that $f_p(n)\le \exp(-c_p n)$ for all $n\in \mb N$. For $n\le N$ this is automatically guaranteed by the choice of $c_p$. So let us now consider $n>N$, and assume that we have already shown $f_p(n')\le \exp(-c_p n')$ for all $n'<n$. Then, using \cref{thm:inductive}, we have
\begin{align*}
f_p(n)&\le\exp(-\delta n)+\eps \sum_{s=1}^{t}p^sf_p(n-s)+(1+\eps)\tau_p \cdot f_p(n-t)+(1+\eps)\sum_{s=t+1}^{n}p^{s}f_p(n-s)\\
&\le\exp(-\delta n)+\eps \sum_{s=1}^{t}p^s\exp(-c_p(n\!-\!s))+(1-4\mu)\exp(-c_p(n\!-\!t))+(1+\eps)\!\!\sum_{s=t+1}^{n}p^{s}\exp(-c_p(n\!-\!s)).
\end{align*}
This yields
\begin{align*}
f_p(n)\cdot \exp(c_pn)&\le \exp(-(\delta-c_p) n)+\eps \sum_{s=1}^{t}p^s\exp(c_p s)+(1-4\mu)\exp(c_pt)+(1+\eps)\sum_{s=t+1}^{n}p^{s}\exp(c_ps)\\
&\le \exp(-\delta n/2)+\eps \sum_{s=1}^{t}p^sp^{-s/2}+(1-3\mu)+(1+\eps)\sum_{s=t+1}^{n}p^{s}p^{-s/2}\\
&\le (1-3\mu)+\exp(-\delta N/2)+\eps \sum_{s=1}^{n}p^{s/2}+\sum_{s=t+1}^{n}p^{s/2}\\
&\le (1-3\mu)+\exp(-\delta N/2)+\frac{\eps}{1-\sqrt{p}}+\frac{p^{t/2}}{1-\sqrt{p}}\le (1-3\mu)+\mu+\mu+\mu=1,
\end{align*}
as desired.
\end{proof}

\section{A relative anticoncentration inequality}

A key ingredient in the proof of \cref{thm:inductive} is the following ``relative anticoncentration inequality'', stating that under certain assumptions, a sum of independent random variables is more anticoncentrated than a ``randomly thinned-out version of itself''. This is similar to \cite[Lemma~7.14]{TV06} and related inequalities of Hal\'asz~\cite{Hal77}, used in previous work on the singularity probability of random matrices.

\begin{theorem}\label{lem:relative-halasz}
Consider independent random variables $X_{1},\dots,X_{n}$, and for each $i=1,\dots,n$, let $X_{i}'$ be an independent copy of $X_{i}$. For some integer $1\le k<n/4$, let $K\subseteq\{1,\dots,n\}$ be a uniformly random subset of exactly $k$ indices (independent from the variables $X_1,\dots,X_n,X_1',\dots, X_n'$). Now, define $X=X_{1}+\dots+X_{n}$ and $X_{K}^{*}=\sum_{i\in K}(X_{i}-X_{i}')$, and assume that $\Pr[X_K^*=0]<1-4k/n$.
Then we have
\[
Q[X]\lesssim\sqrt{\frac{k}{n}}\cdot \frac{\Pr[X_{K}^{*}=0]}{1-\Pr[X_{K}^{*}=0]}.
\]
\end{theorem}
The main ingredient to prove \cref{lem:relative-halasz} is the following (special case of a) theorem of Kesten~\cite{Kes69}.
\begin{theorem}[{\cite[Theorem~2]{Kes69}}]
\label{thm:kesten}Consider independent random variables $X_{1},\dots,X_{m}$ and $0<p<1$, such that $Q[X_i]\le p$ holds for each $i=1,\dots,m$. Then
\[
Q[X_1+\dots+X_m]\lesssim\frac{p}{\sqrt{(1-p)m}}.
\]
\end{theorem}
We note that \cref{thm:kesten} is the only part of this paper that is not self-contained. It was proved with a simple (but quite ingenious!) combinatorial argument, via an anticoncentration inequality of Rogozin~\cite{Rog61} (which, in this context, is essentially equivalent to the so-called \emph{Erd\H os--Littlewood--Offord theorem}~\cite{Erd45}). See also \cite{PY77} for an alternative Fourier-analytic proof.

In our proof of \cref{lem:relative-halasz}, we will actually need the following slight generalisation of \cref{thm:kesten}, which compares $Q[X_1+\dots+X_m]$ to probabilities that the random variables $X_i$ are equal to independent copies of themselves.

\begin{corollary}
\label{cor:relative-halasz-0}Consider independent random variables
$X_{1},\dots,X_{n}$ and $0<p<1$. For each $i=1,\dots,n$, let $X_{i}'$ be an independent
copy of $X_{i}$. Suppose that we have $\Pr[X_{i}=X_{i}']\le p$ for
at least $m\ge 2$ different indices $i\in \{1,\dots,n\}$. Then
\[
Q[X_1+\dots+X_n]\lesssim\frac{p}{\sqrt{(1-p)m}}.
\]
\end{corollary}

To deduce \cref{cor:relative-halasz-0} from \cref{thm:kesten} we will need a few more general inequalities. First, we record the (near-trivial) fact that a sum of independent random variables is at least as anticoncentrated as its summands.

\begin{fact}\label{fact:anticoncentration-monotonicity}
    For independent random variables $X,Y$ we have $Q[X+Y]\le Q[X]$.
\end{fact}
\begin{proof}
    For any $z\in \mb R$ we have
    \[\Pr[X+Y=z]=\Pr[X=z-Y]=\mb E\big[\Pr[X=z-Y\,|\,Y]\big]\le \mb E\big[Q[X]\big]=Q[X].\qedhere\]
\end{proof}

Next, we record a simple ``duplication'' inequality for a sum of independent random variables.

\begin{lemma}
\label{lem:duplication-0}Consider independent discrete
random variables $X_1$ and $X_2$. Letting $X_1'$ and $X_2'$ be independent copies of $X_1$ and $X_2$, we have
\[
Q[X_1+X_2]\le \Pr[X_1=X_1']^{1/2}\cdot \Pr[X_2=X_2']^{1/2}.
\]
In particular (applying the above inequality to the case where $X_2$ has the distribution of $-X_1$), we have
\[
Q[X_1-X_1']\le \Pr[X_1=X_1'].
\]
\end{lemma}
\begin{proof}
We have
\[
Q[X_{1}+X_{2}] =\sup_{x}\sum_{z}\Pr[X_{1}=z]\Pr[z+X_{2}=x] \le\sup_{x}\left(\sum_{z}\Pr[X_{1}=z]^{2}\right)^{1/2}\left(\sum_{z}\Pr[z+X_{2}=x]^{2}\right)^{1/2}
\]
by the Cauchy--Schwarz inequality. But note that (for any $x\in \mb R$)
\begin{align*}
\sum_{z}\Pr[X_1=z]^{2} & =\sum_{z}\Pr[X_1=z]\Pr[X_1'=z]=\Pr[X_1=X_1'],\\
\sum_{z}\Pr[z+X_2=x]^{2} & =\sum_{z}\Pr[X_2=x-z]\Pr[X_2'=x-z]=\Pr[X_2=X_2'],
\end{align*}
and the desired result follows.
\end{proof}

Now, we are ready to prove \cref{cor:relative-halasz-0}.

\begin{proof}[Proof of \cref{cor:relative-halasz-0}]
We may assume that $\Pr[X_{i}=X_{i}']\le p$ for $i=1,\dots,m$. By \cref{fact:anticoncentration-monotonicity} we may furthermore assume that $n=m$ and that $m\ge 2$ is even. Also, note that we may assume that $X_1,\dots,X_m$ are discrete random variables\footnote{Formally: for each $i$, write $A_i=\{x:\Pr[X_i=x]>0\}$ as the ``atomic part'' of $X_i$, and note that for any $z$ we have $\Pr[X_1+\dots+X_n=z\text{ and }X_i\notin A_i\text{ for some }i]=0$. So, if we alter the behaviour of $X_i$ when $X_i\notin A_i$ (to make $X_i$ discrete), this cannot decrease $Q[X_1+\dots+X_m]$. We can make these alterations in a way that has an arbitrarily small impact on probabilities of the form $\Pr[X_i=X_i']$.}.

By the first part of \cref{lem:duplication-0}, we see that $Q[X_1+\dots+X_m]$ is at most
\[\Pr[(X_1+\dots+X_{m/2})=(X_1'+\dots+X_{m/2}')]^{1/2}\cdot \Pr[(X_{m/2+1}+\dots+X_{m})=(X_{m/2+1}'+\dots+X_{m}')]^{1/2}.\]
Then, note that 
\begin{align*}\Pr[(X_1+\dots+X_{m/2})=(X_1'+\dots+X_{m/2}')]&=\Pr[(X_1-X_1')+\dots+(X_{m/2}-X_{m/2}')=0]\\
&\le Q[(X_1-X_1')+\dots+(X_{m/2}-X_{m/2}')].\end{align*}
By the second part of \cref{lem:duplication-0}, for each $i=1,\dots,m$, we have $Q[X_i-X_i']\le \Pr[X_i=X_i']\le p$, so \cref{thm:kesten} yields
\[\Pr[(X_1+\dots+X_{m/2})=(X_1'+\dots+X_{m/2}')]\lesssim \frac{p}{\sqrt{(1-p)(m/2)}}\lesssim \frac{p}{\sqrt{(1-p)m}}.\]
By exactly the same argument, we obtain
\[\Pr[(X_{m/2+1}+\dots+X_{m})=(X_{m/2+1}'+\dots+X_{m}')]\lesssim \frac{p}{\sqrt{(1-p)m}},\]
and the desired result follows.
\end{proof}
Finally, we use \cref{cor:relative-halasz-0} to prove \cref{lem:relative-halasz}.

\begin{proof}[Proof of \cref{lem:relative-halasz}]
Let $h=\lfloor n/k\rfloor$ and let $I_{1},\dots,I_{h}\in\{1,\dots,n\}$
be uniformly random disjoint subsets of exactly $k$ indices (independent from all other random objects in the statement of \cref{lem:relative-halasz}). For
each $j=1,\dots,h$, define 
\[
X_{j}^{*}=\sum_{i\in I_{j}}(X_{i}-X_{i}'),\quad\rho_{j}=\Pr[X_{j}^{*}=0\,|\,I_{j}]=\Pr\Big[\sum_{i\in I_{j}}X_{i}=\sum_{i\in I_{j}}X_{i}'\,\Big|\,I_{j}\Big],
\]
and let $\lambda=1-\Pr[X_K^*=0]>4k/n>2/h$. Note that for each $j=1,\dots,h$, the random variable $X_j^*$ on its own has exactly the same distribution as $X_K^*$. Thus, we have $\E[\rho_{1} +\dots+ \rho_{h}]=h\cdot \Pr[X_K^*=0]=(1-\lambda)h$. So, we can fix an outcome of $I_1,\dots,I_h$ satisfying $\rho_{1}+\dots+\rho_{h}\le (1-\lambda)h$.

Now, we can have at most $(1-\lambda/2)h$ indices $j\in \{1,\dots,h\}$ with $\rho_{j}\ge (1-\lambda)/(1-\lambda/2)$. Thus, there are at least $(\lambda/2)h$ indices $j\in \{1,\dots,h\}$ satisfying $\rho_{j}< (1-\lambda)/(1-\lambda/2)$. 
Then, \cref{fact:anticoncentration-monotonicity} and \cref{cor:relative-halasz-0} (applied with $m=\lceil (\lambda/2)h\rceil \ge 2$ and $p=(1-\lambda)/(1-\lambda/2)<1-\lambda/2$, with the random variables $\sum_{i\in I_{j}}X_{i}$ and their independent copies $\sum_{i\in I_{j}}X_{i}'$ for $j=1,\dots,h$) give \[Q[X] \le Q\Big[\sum_{i\in I_{1}}X_{i}+\dots+\sum_{i\in I_{h}}X_{i}\Big] \lesssim \frac{(1-\lambda)/(1-\lambda/2)}{\sqrt{(\lambda/2)\cdot (\lambda/2) h}} \lesssim  \frac{1}{\sqrt{h}}\cdot \frac{1-\lambda}{\lambda}\lesssim \sqrt{\frac{k}{n}}\cdot \frac{\Pr[X_K^*=0]}{1-\Pr[X_K^*=0]}.\qedhere\]
\end{proof}
Note that \cref{lem:relative-halasz} only applies under the assumption that $\Pr[X_K^*=0]$ is sufficiently far away from 1. We record a corollary with an additional condition that allows us to easily verify this assumption.

\begin{corollary}\label{cor:relative-assumption}
Consider $\gamma>0$ and $0<p<1$, and in the setting of \cref{lem:relative-halasz}, assume that $2/\gamma\le k<(1-p)n/8$ and that we have $Q[X_{i}]\le p$ for at least $\gamma n$ indices $i\in \{1,\dots,n\}$. Then $\Pr[X^*_K=0]\le  (1+p)/2< 1-4k/n$ and consequently
\[
Q[X]\lesssim\sqrt{\frac{k}{n}}\cdot \frac{\Pr[X_{K}^{*}=0]}{1-\Pr[X_{K}^{*}=0]}\lesssim_p \sqrt{\frac{k}{n}}\cdot \Pr[X_{K}^{*}=0].
\]

\end{corollary}

\begin{proof}
Let $I$ be a set of $\lceil \gamma n\rceil$ indices $i\in \{1,\dots,n\}$ satisfying $Q[X_{i}]\le p$. Then $|K\cap I|$ has a hypergeometric distribution, with mean $\mathbb{E}[|K\cap I|]=|I|\cdot k/n\ge \gamma k$ and variance $\operatorname{Var}[|K\cap I|]\le  |I|\cdot k/n=\mathbb{E}[|K\cap I|]$. So by Chebyshev's inequality, we have $\Pr[K\cap I=\emptyset]\le 1/\mathbb{E}[|K\cap I|]\le 1/(\gamma k)\le 1/2$. Thus, we have $K\cap I\ne \emptyset$ with probability at least $1/2$. For any such outcome of $K$ we have $\Pr[X^*_K=0\,|\,K]\le p$ by \cref{fact:anticoncentration-monotonicity}. So, we can conclude that $\Pr[X^*_K=0]\le (1+p)/2$. Using $k<(1-p)n/8$, it is easy to check that $(1+p)/2< 1-4k/n$, and the rest of the statement is obtained by applying \cref{lem:relative-halasz}.
\end{proof}

\section{Proof of the recursive bound}

In this section we prove \cref{thm:inductive} (and along the way, \cref{prop:easy}). So, let $0<p<1$, and let $A$ be as in \cref{def:A}. Throughout, we will assume without loss of generality that the last $T$ columns of $A_{\mathrm{fix}}$
induce a $T\times T$ submatrix with nonzero permanent. We will also use the following
notation.
\begin{definition}
In the setting of \cref{def:A}:
\begin{compactitem}
\item For a set $I\subseteq\{1,\dots,n+T\}$, let $-I=\{1,\dots,n+T\}\setminus I$
be the complement of $I$.
\item For $i\in I$, let $I-i=I\setminus\{i\}$, and for $i\notin I$, let
$I+i=I\cup\{i\}$.
\item For a set $J\subseteq\{1,\dots,n+T\}$, let $A[J]$ be the $|J|\times|J|$
submatrix of $A$ indexed by the first $|J|$ rows and the columns
indexed by $J$.
\item For $s\in\{1,\dots,n\}$, let $A^{\uparrow s}$ be the $(n+T-s)\times (n+T)$ submatrix of $A$ consisting of the first $n+T-s$ rows.
\item For $s\in\{1,\dots,n\}$, $z\in\mb R$ and $\alpha>0$, let $\mathcal{E}_{z}(s,\alpha)$
be the event that for more than $\alpha\binom{n}{s}$ of the $\binom{n}{s}$
size-$s$ subsets $I\subseteq\{1,\dots,n\}$, we have $\on{per}(A[-I])\ne z$.
We also use the shorthand $\mathcal{E}(s,\alpha)=\mathcal{E}_{0}(s,\alpha)$
and $\mathcal{E}_{z}(s)=\mathcal{E}_{z}(s,0)$ and $\mathcal{E}(s)=\mathcal{E}_{0}(s,0)$.
\end{compactitem}
\end{definition}

Note that the event $\mathcal{E}(s)$ precisely means that the $(n+T-s)\times (n+T)$ matrix $A^{\uparrow s}$ contains some $(n+T-s)\times (n+T-s)$ submatrix with non-zero permanent. In particular, $\mathcal{E}(n)$ does not depend on the randomness of $A_{\mathrm{rand}}$ (since $A^{\uparrow n}=A_{\mathrm{fix}}$), and holds deterministically by the conditions in \cref{def:A}.

The permanent enjoys an analogue of the minor expansion formula for the determinant, which we record below. This fact allows us to study permanents recursively.
\begin{fact}\label{fact:minor-expansion}
Fix a set $J\subseteq\{1,\dots,n\}$ of size $s$. We have
\[
\on{per}(A[-J])=\sum_{\substack{i\in \{1,\dots,n+T\}\\i\notin J}}\xi_{i}\on{per}(A[-(J+i)]),
\]
where $\xi_{1},\dots,\xi_{n+T}$ are the entries of the $(n+T-s)$-th
row of $A$.
\end{fact}

We collect various estimates on probabilities
involving the events $\mathcal{E}_{z}(s,\alpha)$. First, the following simple lemma gives a bound on the probability that $\on{per}(A[-J])=z$ for all size-$(s-1)$ subsets $J\subseteq\{1,\dots,n\}$, given that there is at least one size-$s$ subset $I\subseteq\{1,\dots,n\}$ with  $\on{per}(A[-I])\ne 0$.

\begin{lemma}
\label{lem:easy-bound}For any $s\in\{1,\dots,n\}$, any $z\in\mb R$
and any outcome of $A^{\uparrow s}$ satisfying $\mathcal{E}(s)$,
we have
\[
\Pr\Big[\overline{\mathcal{E}_{z}(s-1)}\,\Big|\,A^{\uparrow s}\Big]\le p^{s}.
\]
\end{lemma}

\begin{proof}
Conditioning on the given outcome of $A^{\uparrow s}$, there is some size-$s$ subset
$I\subseteq\{1,\dots,n\}$ satisfying $\on{per}(A[-I])\ne0$. Write $\xi_{1},\dots,\xi_{n+T}$
for the entries of the $(n-s+1)$-th row of $A_{\mathrm{rand}}$,
and reveal $\xi_{i}$ for all $i\in \{1,\dots,n+T\}\setminus I$.

Now, consider all the $s$ permanents of the form $\on{per}(A[-(I-i)])$ for $i\in I$.
It suffices to show that (conditioned on the revealed information, using
only the randomness of $\xi_{i}$ for $i\in I$), the probability
that all these permanents are equal to $z$ is at most $p^{s}$. This
is a fairly immediate consequence of \cref{fact:minor-expansion}, using that $Q[\xi_{i}]\le p$ for all $i\in I$. Indeed, in our conditional probability space, for each $i\in I$ the
random variable $\on{per}(A[-(I-i)])$ only depends on the randomness
of $\xi_{i}$, and at most one of the outcomes of $\xi_{i}$
will cause $\on{per}(A[-(I-i)])\ne z$ (since $\on{per}(A[-I])\ne0$).
\end{proof}

Note that \cref{lem:easy-bound} in particular implies $\Pr[\overline{\mathcal{E}_{z}(s-1)}\,|\,\mathcal{E}(s)]\le p^{s}$ for any $s\in\{1,\dots,n\}$ and any $z\in\mb R$. This is already enough to prove \cref{prop:easy}, as follows.

\begin{proof}[Proof of \cref{prop:easy}]
Consider any $z\in\mb R$.
Recalling that the event $\mathcal{E}(n)$ always holds, we have
\[
\Pr[\on{per}(A)\ne z] =\Pr[\mathcal{E}_{z}(0)]  \ge\Pr\left[\mc E_z(0)\cap \bigcap_{s=1}^n\mathcal{E}(s)\right]=\Pr[\mc E_z(0)\,|\,\mc E(1)]\prod_{s=2}^{n}\Pr[\mathcal{E}(s-1)\,|\,\mathcal{E}(s)]\ge\prod_{s=1}^{n}(1-p^{s}),
\]
by \cref{lem:easy-bound}. The desired result follows.
\end{proof}
The next lemma is a simple application of Markov's inequality. 
\begin{lemma}
\label{lem:markov}For any $s\in\{1,\dots,n\}$ and $0<\alpha<1/2$, we
have 
\[
\Pr\Big[\overline{\mathcal{E}(s,\alpha)}\Big]\le\frac{f_p(n-s)}{1-\alpha}\le(1+2\alpha)f_p(n-s).
\]
\end{lemma}

\begin{proof}
If $\overline{\mathcal{E}(s,\alpha)}$ holds, there must be at least $(1-\alpha)\binom{n}{s}$ size-$s$ subsets $I\subseteq\{1,\dots,n\}$ with $\on{per}(A[-I])=0$. For each size-$s$ subset $I\subseteq\{1,\dots,n\}$, the matrix $A[-I]$ has the form in \cref{def:A} (with $n-s$ random rows), so the expected number of such subsets $I\subseteq\{1,\dots,n\}$ satisfying
$\on{per}(A[-I])=0$ is at most $f_p(n-s)\binom{n}{s}$. The desired
result then follows from Markov's inequality.
\end{proof}
Finally, the next lemma is where most of the difficulty lies. It combines the ideas from \cref{lem:easy-bound,lem:markov}, together with \cref{lem:relative-halasz} (our relative anticoncentration inequality).
\begin{lemma}
\label{lem:hard-bound}
Fix $\alpha>0$. For large $n$, any $s\in\{1,\dots,n\}$, and any $z\in\mb R$, we have
\[
\Pr\Big[\overline{\mathcal{E}_{z}(s-1)}\cap\mathcal{E}(s,\alpha)\Big]
\le \Pr\Big[\overline{\mathcal{E}_{z}(s-1,\alpha/(4s))}\cap\mathcal{E}(s,\alpha)\Big]\le \alpha p^{s}f_p(n-s)+\exp(-\Omega_{\alpha,p}(n)).
\]
\end{lemma}

In our proof of \cref{lem:hard-bound} we will use the following simple fact about dense set families.

\begin{lemma}
\label{lem:set-family-new}For $n\in\mathbb N$, $s\in \{1,\dots,n\}$ and $\alpha>0$, let $\mathcal{F}$ be a collection of at least
$\alpha\binom{n}{s}$ size-$s$ subsets of $\{1,\dots,n\}$. Then there is a collection $\mathcal{H}$ of $|\mathcal{H}|\ge(\alpha/2)s\binom{n}{s}$ pairs of the form $(F,G)$ with $F\in \mathcal{F}$ and $G\subseteq F$ of size $|G|=s-1$, such that for each $(F,G)\in \mathcal{H}$, there are at least $(\alpha/2)(n-s+1)$ sets $F'\in \mathcal{F}$ satisfying $(F',G)\in \mathcal{H}$.
\end{lemma}
\begin{proof}
Consider the bipartite graph whose vertices on the left side correspond to the sets $F\in \mathcal{F}$ and whose vertices on the right side correspond to  the $\binom{n}{s-1}$ different size-$(s-1)$ subsets $G\subseteq \{1,\dots,n\}$. We draw an edge between a vertex $F$ on the left and a vertex $G$ on the right if and only if $G\subseteq F$. Then the graph has $s\cdot |\mathcal{F}|\ge \alpha s\binom{n}{s}$ edges. There can be at most $\binom{n}{s-1}\cdot (\alpha/2)(n-s+1)=(\alpha/2)s\binom{n}{s}$ edges $(F,G)$ with $\deg(G)\le (\alpha/2)(n-s+1)$. Thus, the graph has at least $(\alpha/2)s\binom{n}{s}$ edges $(F,G)$ with $\deg(G)> (\alpha/2)(n-s+1)$. We can now take $\mathcal{H}$ to be the set of these edges.
\end{proof}

Now we prove \cref{lem:hard-bound}.

\begin{proof}[Proof of \cref{lem:hard-bound}]Let $C_p$ be a constant depending only on $p$ such that in the conclusion of \cref{cor:relative-assumption} we always have $Q[X]\le C_p \cdot \sqrt{k/n}\cdot \Pr[X_{K}^{*}=0]$. Let us fix some small $\nu>0$ depending on $\alpha$ and $p$ such that $\nu< (1-p)/8$ and $\sqrt{\nu}\le (\alpha^2/4)\cdot (p/C_p)$.

The first inequality in the statement of \cref{lem:hard-bound} follows from the fact that the event $\mathcal{E}_{z}(s-1)=\mathcal{E}_{z}(s-1,0)$ contains the event $\mathcal{E}_{z}(s-1,\alpha/(4s))$. So it suffices to prove the second inequality. Due to the asymptotic term $\exp(-\Omega_{\alpha,p}(n))$ in the statement, we may assume that $n$ is large with respect to $p,\alpha,\nu$.

Let $\xi_{1},\dots,\xi_{n+T}$ be the entries of the $(n+T-s+1)$-th row of $A$. 
Our strategy is to use \cref{cor:relative-assumption} to estimate the probability
of the event $\overline{\mathcal{E}_{z}(s-1,\alpha/(4s))}$, conditioned on an outcome
of $A^{\uparrow s}$ such that $\mathcal{E}(s,\alpha)$ holds (using the randomness of $\xi_{1},\dots,\xi_{n+T}$). Then, we will sum this estimate over all outcomes of $A^{\uparrow s}$ satisfying $\mathcal{E}(s,\alpha)$ to obtain a bound for the probability of $\overline{\mathcal{E}_{z}(s-1,\alpha/(4s))}\cap \mc E(s,\alpha)$.  In order to bound the probability of the event $\overline{\mathcal{E}_{z}(s-1,\alpha/(4s))}$ conditioned on an outcome of $A^{\uparrow s}$, we will consider the events $\on{per}(A[-I])\ne0$ and $\on{per}(A[-J])\ne0$ for various choices of nested subsets $J\su I\su \{1,\dots,n\}$ of sizes $|I|=s$ and $|J|=s-1$. Let $(I^*,J^*)$ be a uniformly random  pair of subsets $J^*\su I^*\su \{1,\dots,n\}$ with $|I^*|=s$ and $|J^*|=s-1$, independent from the random matrix $A$ (note that the total number of such subset pairs is exactly $s\binom{n}{s}$).

\medskip\noindent\textbf{Step 1: Reducing to a linear anticoncentration problem.} Consider an outcome of $A^{\uparrow s}$ such that $\mathcal{E}(s,\alpha)$ holds. Let $\mathcal{F}$ be the family of all size-$s$ subsets $I\subseteq \{1,\dots,n\}$ with $\on{per}(A[-I])\ne0$ (note that this family $\mathcal{F}$ is determined by the outcome of $A^{\uparrow s}$). By the definition of the event $\mathcal{E}(s,\alpha)=\mathcal{E}_0(s,\alpha)$, we have $|\mathcal{F}|> \alpha\binom{n}{s}$. Thus, by \cref{lem:set-family-new} we can find a collection $\mathcal{H}$ of at least $(\alpha/2)s\binom{n}{s}$ pairs $(I,J)$ with $I\in \mathcal{F}$ and $J\su I$ of size $|J|=s-1$, such that for each $(I,J)\in \mathcal{H}$, there are at least $(\alpha/2)(n-s+1)$ sets $I'\in \mathcal{F}$ satisfying $(I',J)\in \mathcal{H}$.

Let us say that a set $I\in \mathcal{F}$ is \emph{parental} if it has a size-$(s-1)$ subset $I'\su I$ with $\on{per}(A[-I'])\ne z$ (we can think of $I'$ as a \emph{child} of $I$). Note that for each size-$(s-1)$ subset $I'\su \{1,\dots,n\}$ there can be at most $n-s+1$ different sets $I\in \mathcal{F}$ with $I'\su I$ (so every child has at most $n-s+1$ parents). In particular, the number of size-$(s-1)$ subsets $I'\su \{1,\dots,n\}$ with $\on{per}(A[-I'])\ne z$ is at least
\[\frac{|\{I\in \mathcal{F}: I\text{ parental}\}|}{n-s+1}\ge \frac{|\{(I,J)\in \mathcal{H}: I\text{ parental}\}|}{s(n-s+1)},\]
observing that for every $I\in \mathcal{F}$ there can be at most $s$ different sets $J$ with $(I,J)\in \mathcal{H}$.

Thus, if for more than half of the pairs $(I,J)\in \mathcal{H}$ the set $I$ is parental, we have more than
\[\frac{|\mathcal{H}|/2}{s(n-s+1)}\ge \frac{(\alpha/4)s\binom{n}{s}}{s(n-s+1)}=\frac{(\alpha/4)(n-s+1)\binom{n}{s-1}}{s(n-s+1)}=\frac{\alpha}{4s}\cdot \binom{n}{s-1}\]
different size-$(s-1)$ subsets $I'\su \{1,\dots,n\}$ with $\on{per}(A[-I'])\ne z$, so the event $\mathcal{E}_{z}(s-1,\alpha/(4s))$ holds.

Therefore, we can conclude \begin{align}
\Pr\Big[\overline{\mathcal{E}_{z}(s-1,\alpha/(4s))}\,\Big|\,A^{\uparrow s}\Big]&\le\Pr\Big[I\text{ is not parental for at least half of the pairs } (I,J)\in \mathcal{H}\,\Big|\,A^{\uparrow s}\Big]\notag \\
&\le\frac{\mb E\big[|\{(I,J)\in \mathcal{H}: I\text{ not parental}\}|\,\big|\,A^{\uparrow s}\big]}{|\mathcal{H}|/2}\notag\\
&=\frac{2}{|\mathcal{H}|}\sum_{(I,J)\in \mathcal{H}}\Pr\Big[I\text{ not parental}\,\Big|\,A^{\uparrow s}\Big]\notag\\
&=2\cdot \Pr\Big[I^{*}\text{ not parental}\,\Big|\,A^{\uparrow s}, \,\mathcal{E}_\mathcal{H}\Big],\label{eq:tricky-markov}
\end{align}
where $\mathcal{E}_\mathcal{H}$ denotes the event that $(I^*,J^{*})\in \mathcal{H}$ (noting that conditional on this event, the pair $(I^*,J^*)$ is a unformly random pair in $\mathcal{H}$). Now, note that
\begin{equation}
\Pr\Big[I^{*}\text{ not parental}\,\Big|\,A^{\uparrow s}, \,\mathcal{E}_\mathcal{H}\Big]\le \Pr\Big[\on{per}(A[-J^{*}])=z\,\Big|\,A^{\uparrow s}, \,\mathcal{E}_\mathcal{H}\Big]\cdot p^{s-1}.\label{eq:bonus-2s}
\end{equation}
The justification of \cref{eq:bonus-2s} is very similar to the proof of \cref{lem:easy-bound} (recalling \cref{fact:minor-expansion}). If we reveal $A^{\uparrow s}$ and an outcome $(I^{*},J^{*})\in \mathcal{H}$, and we reveal $\xi_{i}$ for $i\in \{1,\dots,n+T\}\setminus J^{*}$, then this already determines $\on{per}(A[-J^{*}])$. If $\on{per}(A[-J^{*}])\ne z$ then $I^{*}$ is automatically parental. Furthermore, in the resulting conditional probability space given this revealed information, for each $i\in J^{*}$ the random variable $\on{per}(A[-(I^{*}-i)])$ only
depends on the randomness of $\xi_{i}$, and at most one of the possible outcomes of $\xi_{i}$ will cause $\on{per}(A[-(I^{*}-i)])= z$ (here we are using that $\on{per}(A[-I^{*}])\ne 0$ as $I^*\in \mathcal{F}$). We need $\on{per}(A[-(I^{*}-i)])= z$ for all $i\in J^{*}$  for $I^{*}$ not to be parental, and we recall our assumption $Q[\xi_{i}]\le p$.

Combining \cref{eq:tricky-markov} and \cref{eq:bonus-2s} yields
\begin{equation}
\Pr\Big[\overline{\mathcal{E}_{z}(s-1,\alpha/(4s))}\,\Big|\,A^{\uparrow s}\Big]\le 2p^{s-1} \cdot \Pr\Big[\on{per}(A[-J^{*}])=z\,\Big|\,A^{\uparrow s}, \,\mathcal{E}_\mathcal{H}\Big]
    \label{eq:comb-of-first-two}
\end{equation}
for any outcome of $A^{\uparrow s}$ satisfying $\mc E(s,\alpha)$.

Note that this in particular implies $\Pr[\overline{\mathcal{E}_{z}(s-1,\alpha/(4s))}\,|\,A^{\uparrow s}]\le 2p^{s-1}$. If $s\ge n/2$, then averaging this over the possible outcomes of $A^{\uparrow s}$ satisfying $\mc E(s,\alpha)$ already yields the bound \[\Pr\Big[\overline{\mathcal{E}_{z}(s-1,\alpha/(4s))}\cap \mc E(s,\alpha)\Big]\le \Pr\Big[\overline{\mathcal{E}_{z}(s-1,\alpha/(4s))}\,\Big|\, \mc E(s,\alpha)\Big]\le 2 p^{s-1}\le 2 p^{n/2-1}=\exp(-\Omega_{p}(n)).\]
Therefore for the rest of the proof we may assume $s<n/2$.

Now, note that if we reveal $A^{\uparrow s}$ and $(I^{*},J^{*})$,
then in the resulting conditional probability space, $\on{per}(A[-J^{*}])$
is a linear function of $\xi_{1},\dots,\xi_{n+T}$ (by \cref{fact:minor-expansion}). For every $i\in \{1,\dots,n+T\}\setminus J^{*}$, the coefficient of $\xi_{i}$ in this linear function is precisely $\on{per}(A[-(J^{*}+i)])$. 
We will apply our relative anticoncentration result in \cref{cor:relative-assumption} to this linear function (or, more precisely, to the linear function obtained by omitting the last $T$ terms).

\medskip\noindent \textbf{Step 2: Applying the relative anticoncentration inequality.}
Let $k=\lfloor \nu (n-s+1)\rfloor$, and let $A^*$ be a random matrix obtained from $A[-J^*]$ as follows. First, for every entry $\xi_i$ in the $(n+T-s+1)$-th row of $A[-J^*]$, replace this entry with $\xi_i-\xi_i'$, where $\xi_i'$ is an independent copy of $\xi_i$. Then, choose a uniformly random subset $K\subseteq\{1,\dots,n\}\setminus J^*$ of exactly $k$ of the first $n-s+1$ column indices of $A[-J^*]$, and set all the entries in the $(n-s+1)$-th row to zero except those indexed by $K$.

\begin{claim}\label{claim:step-2}
For any fixed outcomes of $A^{\uparrow s}$ and $(I^{
*},J^{*})\in\mathcal{H}$, we have
\[\Pr\Big[\on{per}(A[-J^{*}])=z\,\Big|\,A^{\uparrow s},(I^{*},J^{*})\Big]\le C_p\cdot  \sqrt{\nu}\cdot \Pr\Big[\on{per}(A^*)=0\,\Big|\,A^{\uparrow s},(I^{*},J^{*})\Big].\]
\end{claim}
\begin{proof}
As in the claim statement, let us fix outcomes of $A^{\uparrow s}$ and $(I^*,J^*)\in \mc{H}$, and consider the conditional probability space obtained by conditioning on these outcomes. 
For a random variable $Y$, we write $Q[Y\,|\,A^{\uparrow s},(I^{*},J^{*})]=\sup_{y}\Pr[Y=y\,|\,A^{\uparrow s},(I^{*},J^{*})]$.

As $(I^{*},J^{*})\in\mathcal{H}$, there are at least $(\alpha/2)(n-s+1)$ different sets $I'\in \mathcal{F}$ with $(I',J^{*})\in\mathcal{H}$. For all of these sets $I'$ we have $J^{*}\su I'$ and $\on{per}(A[-I'])\ne 0$. In other words, there are at least $(\alpha/2)(n-s+1)$ indices $i\in \{1,\dots,n\}\setminus J^{*}$ such that $\on{per}(A[-(J^{*}+i)])\ne 0$.

Now, let us define $X_i=\on{per}(A[-(J^{*}+i)])\cdot \xi_{i}$ for all $i\in \{1,\dots,n+T\}\setminus J^{*}$. Here, the coefficient $\on{per}(A[-(J^{*}+j)])$ is already determined by $A^{\uparrow s}$ and $(I^{*},J^{*})$, so after revealing $A^{\uparrow s}$ and $(I^{*},J^{*})$, the random variable $X_i$ is a rescaling of $\xi_{i}$. Thus, the random variables $X_i$ for $i\in \{1,\dots,n+T\}\setminus J^{*}$ are independent in the conditional conditional probability space obtained by revealing $A^{\uparrow s}$ and $(I^{*},J^{*})$. Furthermore, by \cref{fact:minor-expansion} we have $\on{per}(A[-J^{*}])=\sum_{i\in \{1,\dots,n+T\}\setminus J^{*}}X_i$. We can now apply \cref{cor:relative-assumption} with $\gamma=\alpha/2$ and the random variables $X_i$ for $i\in \{1,\dots,n\}\setminus J^{*}$. At least $(\alpha/2)(n-s+1)$ of these random variables satisfy $Q[X_i\,|\,A^{\uparrow s},(I^{*},J^{*})]\le p$, since this holds for all indices $i\in \{1,\dots,n\}\setminus J^{*}$ such that $\on{per}(A[-(J^{*}+i)])\ne 0$ (also noting that $k=\lfloor \nu(n-s+1)\rfloor\ge \lfloor \nu n/2\rfloor\ge 4/\alpha$ if $n$ is sufficiently large in terms of $p$, $\alpha$ and $\nu$). This yields
\begin{align*}
Q\Big[\on{per}(A[-J^{*}])\,\Big|\,A^{\uparrow s},(I^{*},J^{*})\Big]&=Q\Bigg[\sum_{i\in \{1,\dots,n+T\}\setminus J^{*}}\!\!\!\!\!\!\!\!X_i\,\Bigg|\,A^{\uparrow s},(I^{*},J^{*})\Bigg]\\
&\le Q\Bigg[\sum_{i\in \{1,\dots,n\}\setminus J^{*}}\!\!\!\!\!\!\!\!X_i\,\Bigg|\,A^{\uparrow s},(I^{*},J^{*})\Bigg]\\
&\le C_p\cdot \sqrt{\frac{k}{n-s+1}}\cdot \Pr\Big[\on{per}(A^*)=0\,\Big|\,A^{\uparrow s},(I^{*},J^{*})\Big].
\end{align*}
Here, for the first inequality we used \cref{fact:anticoncentration-monotonicity}, and for the second inequality we used \cref{cor:relative-assumption}, noting that $\on{per}(A^*)=\sum_{i\in K}\on{per}(A[-(J^{*}+i)])(\xi_{i}-\xi'_{i})=\sum_{i\in K} (X_i-X'_i)$ if we define $X'_i=\on{per}(A[-(J^{*}+i)])\cdot \xi'_{i}$ for $i\in \{1,\dots,n\}\setminus J^{*}$ (then $X'_i$ is an independent copy of $X_i$ in this conditional probability space). Recalling $k\le \nu (n-s+1)$, we obtain the desired inequality.
\end{proof}
Averaging the inequality in \cref{claim:step-2} over all possible outcomes of $(I^{*},J^{*})\in \mathcal{H}$ yields
\begin{equation}
\Pr\Big[\on{per}(A[-J^{*}])=z\,\Big|\,A^{\uparrow s},\, \mathcal{E_H}\Big]\le C_p\cdot \sqrt{\nu}\cdot \Pr\Big[\on{per}(A^*)=0\,\Big|\,A^{\uparrow s},\, \mathcal{E_H}\Big].
\label{eq:initial-conditioning}
\end{equation}
Now, for every fixed outcome of $A^{\uparrow s}$ satisfying $\mathcal{E}(s,\alpha)$, recall that $|\mathcal{H}|\ge (\alpha/2)s\binom{n}{s}$, and that $(I^*,J^*)$ is independent from $A^{\uparrow s}$. This implies
\[\Pr\Big[\mathcal{E_H}\,\Big|\,A^{\uparrow s}\Big]=\Pr\Big[(I^{*},J^{*})\in \mathcal{H}\,\Big|\,A^{\uparrow s}\Big]= \frac{|\mathcal{H}|}{s\binom{n}{s}}\ge \alpha/2.\]
Therefore we obtain
\[
\Pr\Big[\on{per}(A^*)=0\,\Big|\,A^{\uparrow s}\Big]\ge \Pr\Big[\on{per}(A^*)=0\text{ and }\mathcal{E_H}\,\Big|\,A^{\uparrow s}\Big]\ge \frac{\alpha}{2}\cdot \Pr\Big[\on{per}(A[-J^{*}])=z\,\Big|\,A^{\uparrow s},\, \mathcal{E_H}\Big].
\]
Combining this with \cref{eq:comb-of-first-two} and \eqref{eq:initial-conditioning} yields
\begin{align*}
\Pr\Big[\overline{\mathcal{E}_{z}(s-1,\alpha/(4s))}\,\Big|\,A^{\uparrow s}\Big]&\le 2C_p\sqrt{\nu} p^{s-1} \cdot \Pr\Big[\on{per}(A^*)=0\,\Big|\,A^{\uparrow s},\, \mathcal{E_H}\Big]\\
&\le \frac{4C_p\sqrt{\nu}}{\alpha} p^{s-1} \cdot \Pr\Big[\on{per}(A^*)=0\,\Big|\,A^{\uparrow s}\Big]\le \alpha p^{s}\cdot  \Pr\Big[\on{per}(A^*)=0\,\Big|\,A^{\uparrow s}\Big]
\end{align*}
for every outcome of $A^{\uparrow s}$ satisfying $\mc E(s,\alpha)$, where for the last inequality we used $\sqrt{\nu}\le (\alpha^2/4)\cdot (p/C_p)$.

Summing over the possible outcomes of $A^{\uparrow s}$ satisfying $\mc E(s,\alpha)$, 
we deduce that
\begin{equation}\label{eq:summary-step-3}
 \Pr\Big[\overline{\mathcal{E}_{z}(s-1,\alpha/(4s))}\cap\mc E(s,\alpha)\Big]\le \alpha p^{s}\cdot \Pr\Big[\on{per}(A^*)=0\cap\mc E(s,\alpha)\Big]\le \alpha p^{s}\cdot \Pr\Big[\on{per}(A^*)=0\Big].
\end{equation}

\medskip\noindent \textbf{Step 3: Recursion.} 
Now, note that the permanent is unaffected by rearranging rows. Let $A'$ be obtained from $A^*$ by moving row $(n+T-s+1)$ (the row we modified to get $A^*$ from $A[-J^*]$) directly after row $T$, to become the new $(T+1)$-th row. Note that the last $T$ entries of the $(n+T-s+1)$-th row of $A^*$ are always zero, hence the last $T$ entries of the $(T+1)$-th row of $A'$ are always zero (and also note that the $T\times T$ submatrix of $A'$ in the first $T$ rows and last $T$ columns has a non-zero permanent, since this is precisely the submatrix in the last $T$ columns of $A_{\mathrm{fix}}$).

Suppose we reveal an outcome of $J^{*}$
and an outcome of the $(T+1)$-th row of $A'$ which is not
all zero. Then, the first $T+1$ rows of $A'$ contain a $(T+1)\times(T+1)$
submatrix with nonzero permanent (in the last $T$ columns, together
with any column for which there is a nonzero entry in the $(T+1)$-th
row). So, conditional on revealed information, $A'$ has the form in \cref{def:A}, with $n-s$ random rows (and $T+1$ fixed rows), so the conditional probability that $\on{per}(A')=0$ is at most $f_p(n-s)$. This shows
\[
\Pr\big[\on{per}(A^*)=0\big]=\Pr\big[\on{per}(A')=0\big]\le \Pr\big[\text{the }(T+1)\text{-th row of }A'\text{ is all-zero}\big]+f_p(n-s).
\]
Now, note that the $(T+1)$-th row of $A'$ is all-zero if and only if all entries of the $(n+T-s+1)$-th row of $A^*$ are zero. The probability of this is at most $p^{k}\le p^{\lfloor \nu n/2\rfloor}=\exp(-\Omega_{p,\alpha}(n))$ (since $Q[\xi_i-\xi'_i]\le Q[\xi_i]\le p$ for every $i\in K$). 
Thus, we obtain
\[\Pr\big[\on{per}(A^*)=0\big]\le f_p(n-s)+\exp(-\Omega_{p,\alpha}(n)),\]
and together with \eqref{eq:summary-step-3} we can conclude
\[\Pr\Big[\overline{\mathcal{E}_{z}(s-1,\alpha/(4s))}\cap \mc E(s,\alpha)\Big]\le \alpha p^{s}\cdot \big(f_p(n-s)+\exp(-\Omega_{\alpha,p}(n))\big).\qedhere\]
\end{proof}
Now, we deduce \cref{thm:inductive}.
\begin{proof}[Proof of \cref{thm:inductive}]
Recalling that the event $\mathcal{E}(n)$ holds deterministically, we can observe that
\begin{equation}
\Pr[\on{per}(A)=z]=\Pr\Big[\overline{\mathcal{E}_{z}(0)}\Big]=\Pr\Big[\overline{\mathcal{E}_{z}(0)}\cap \mathcal{E}(n)\Big]\le\Pr\Big[\overline{\mathcal{E}_{z}(0)}\cap\mathcal{E}(t)\Big]+\sum_{s=t+1}^{n}\Pr\Big[\overline{\mathcal{E}(s-1)}\cap\mathcal{E}(s)\Big].\label{eq:permanent-event-breakdown}
\end{equation}
We can further break down the above terms, as follows. Let $\alpha=\alpha_t=\eps/3$, and for all $1\le s\le t$ let $\alpha_{s-1}=\alpha_s/(4s)$. Then for $s=t+1,\dots,n$, we have
\begin{equation}
\Pr\Big[\overline{\mathcal{E}(s-1)}\cap\mathcal{E}(s)\Big]  =\Pr\Big[\overline{\mathcal{E}(s-1)}\cap(\mathcal{E}(s)\setminus\mathcal{E}(s,\alpha))\Big]+\Pr\Big[\overline{\mathcal{E}(s-1)}\cap\mathcal{E}(s,\alpha)\Big],\label{eq:row-event-breakdown}\\
\end{equation}
and furthermore
\begin{align}
\Pr\Big[\overline{\mathcal{E}_{z}(0)}\cap\mathcal{E}(t)\Big]&\le\Pr\Big[\overline{\mathcal{E}_{z}(0)}\cap(\mathcal{E}(t)\backslash\mathcal{E}(t,\alpha))\Big] +\Pr\Big[\overline{\mathcal{E}_{z}(0)}\cap \mathcal{E}(t,\alpha)\Big]\notag\\
&\le \Pr\Big[\overline{\mathcal{E}_{z}(0)}\cap(\mathcal{E}(t)\backslash\mathcal{E}(t,\alpha))\Big]\!+\!\Pr\Big[\overline{\mathcal{E}_z(0)}\cap\mathcal{E}(1,\alpha_1)\Big]\!+\!\sum_{s=2}^t \Pr\Big[\overline{\mathcal{E}(s\!-\!1,\alpha_{s\!-\!1})}\cap\mathcal{E}(s,\alpha_s)\Big]
\label{eq:last-rows-breakdown}
\end{align}
holds. We can directly substitute \cref{lem:hard-bound} into some of the
terms in \cref{eq:row-event-breakdown} and \cref{eq:last-rows-breakdown}. To bound the other terms, let us write $\sup_{\mathcal{E}}\Pr[\mathcal{F}\,|\,A^{\uparrow s}]$
and $\inf_{\mathcal{E}}\Pr[\mathcal{F}\,|\,A^{\uparrow s}]$ to
indicate a supremum or infimum over all outcomes of $A^{\uparrow s}$ satisfying an event $\mathcal{E}$. Then, by \cref{lem:easy-bound,lem:markov}, for $s=t+1,\dots,n$ we obtain
\begin{align}
\Pr\Big[\overline{\mathcal{E}(s-1)}\cap(\mathcal{E}(s)\setminus\mathcal{E}(s,\alpha))\Big] & \le\Pr\Big[\overline{\mathcal{E}(s,\alpha)}\Big]\cdot\Pr\Big[\overline{\mathcal{E}(s-1)}\,\Big|\,\mathcal{E}(s)\setminus\mathcal{E}(s,\alpha)\Big]\nonumber\\
&\le \Pr\Big[\overline{\mathcal{E}(s,\alpha)}\Big]\cdot\sup_{\mathcal{E}(s)}\Pr\Big[\overline{\mathcal{E}(s-1)}\,\Big|\,A^{\uparrow s}\Big]  \le(1+2\alpha)f_p(n-s)\cdot p^{s},\label{eq:sup-inequality}
\end{align}
as well as (recalling the definition of $\tau_p$ in \cref{eq:def:tau-p})
\begin{align}
\Pr\Big[\overline{\mathcal{E}_{z}(0)}\cap(\mathcal{E}(t)\backslash\mathcal{E}(t,\alpha))\Big] & \le\Pr\Big[\overline{\mathcal{E}(t,\alpha)}\Big]\cdot\Bigg(1-\Pr\Big[\mathcal{E}_z(0)\,\Big|\,\mathcal{E}(t)\setminus\mathcal{E}(t,\alpha)\Big]\Bigg)\nonumber \\
&\le\Pr\Big[\overline{\mathcal{E}(t,\alpha)}\Big]\cdot\Bigg(1-\inf_{\mathcal{E}(1)}\Pr\Big[\mathcal{E}_z(0)\,\Big|\,A^{\uparrow s}\Big]\cdot \prod_{s=2}^t\inf_{\mathcal{E}(s)}\Pr\Big[\mathcal{E}(s-1)\,\Big|\,A^{\uparrow s}\Big]\Bigg)\nonumber \\
 & \le(1+2\alpha)f_p(n-t)\cdot\Bigg(1-\prod_{s=1}^t(1-p^s)\Bigg)\le(1+2\alpha)\tau_p\cdot f_p(n-t)\label{eq:sup-inequality-2}.
\end{align}
Substituting \cref{lem:hard-bound} and \cref{eq:sup-inequality} into \cref{eq:row-event-breakdown}, for $s=t+1,\dots,n$, we obtain
\[
\Pr\Big[\overline{\mathcal{E}(s-1)}\cap\mathcal{E}(s)\Big]\le(1+3\alpha)p^sf_p(n-s)+\exp(-\Omega_{\eps,p}(n)),
\]
and substituting \cref{lem:hard-bound} and \cref{eq:sup-inequality-2} into \cref{eq:last-rows-breakdown}, we obtain
\[
\Pr\Big[\overline{\mathcal{E}_{z}(0)}\cap\mathcal{E}(t)\Big]\le(1+2\alpha)\tau_p \cdot f_p(n-t)+\sum_{s=1}^t\alpha p^s f_p(n-s)+\exp(-\Omega_{\eps,p,t}(n)).
\]
Substituting the above two estimates into \cref{eq:permanent-event-breakdown} yields
\[\Pr[\on{per}(A)=z]\le (1+2\alpha)\tau_p\cdot f_p(n-t)+\alpha\sum_{s=1}^t p^s f_p(n-s)+(1+3\alpha)\sum_{s=t+1}^n p^s f_p(n-s)+\exp(-\Omega_{\eps,p,t}(n)),\]
as desired (recalling $\alpha=\eps/3$).
\end{proof}

\bibliographystyle{amsplain_initials_nobysame_nomr}
\bibliography{main}

\end{document}